\documentclass[11pt,oneside,a4paper]{article}
\usepackage{hyperref}
\usepackage{subcaption}
\usepackage{graphicx}
\usepackage{amsfonts}
\usepackage{graphicx}
\usepackage{epstopdf}
\usepackage{algorithmic}
\usepackage{mismath}
\usepackage{cancel}
\usepackage{amsmath}
\usepackage{amsthm}
\usepackage{float}

\usepackage{amssymb}
\usepackage{todonotes}
\newcommand{\dd}{\mathrm{d}}

\newcommand{\demi}{\frac{1}{2}}
\newcommand{\m}{\mathbf{m}}
\newcommand{\yu}{\mathbf{u}}

\newcommand{\ro}{\varrho}
\newcommand{\rhoti}{\tilde{\rho}}
\newcommand{\yuti}{\tilde{\yu}}

\newcommand{\real}{\mathbb{R}}

\newcommand{\nat}{\mathbb{N}}
\newcommand{\TT}{{\mathbb{T}^2}}

\newcommand{\Ensk}{\mathcal{E}_{\textnormal{NSK}}}
\newcommand{\ENSK}{E_{\textnormal{NSK}}}
\newcommand{\Ld}{\Delta_h}

\newcommand{\fx}{\nabla^+_x}
\newcommand{\fy}{\nabla^+_y}
\newcommand{\cx}{\nabla^c_x}
\newcommand{\cy}{\nabla^c_y}
\newcommand{\bx}{\nabla^-_x}
\newcommand{\by}{\nabla^-_y}

\newcommand{\bm}{\mathbf{m}}
\newcommand{\convws}{\stackrel{\ast}{\rightharpoonup}}
\newcommand{\ppsi}{\tilde{\boldsymbol{\psi}}}

\ifpdf
\DeclareGraphicsExtensions{.eps,.pdf,.png,.jpg}
\else
\DeclareGraphicsExtensions{.eps}
\fi

\title{Convergence of a Finite Volume Scheme for the Navier-Stokes-Korteweg Model via Dissipative Solutions
	\thanks{All authors were supported by the DFG within the priority research program SPP 2410, project 
		525866748. 
}}
\author{Jan Giesselmann\thanks{Department of Mathematics, Technische Universität Darmstadt, Germany ({giesselmann@mathematik.tu-darmstadt.de})}
	\and  Philipp \"Offner\thanks{Institute of Mathematics, Clausthal University of Technology, Germany, ({philipp.oeffner@tu-clausthal.de})}
	\and Robert Sauerborn\thanks{Institute of Mathematics, Clausthal University of Technology, Germany, ({robert.sauerborn@tu-clausthal.de})}
}

	\newcommand{\bu}{\mathbf{u}}
	\newcommand{\roti}{\tilde{\ro}}

\newcommand{\be}{\boldsymbol{e}}
\newcommand{\Dhpx}{D^+_{h,x}}
\newcommand{\Dhpy}{D^+_{h,y}}

\newcommand{\Dhmx}{D^-_{h,x}}
\newcommand{\Dhmy}{D^-_{h,y}}

\newcommand{\Dhcx}{D^c_{h,x}}
\newcommand{\Dhcy}{D^c_{h,y}}
\theoremstyle{definition}
\newtheorem{definition}{Definition}
\newtheorem{assumption}[definition]{Assumption}

\newtheorem{lemma}[definition]{Lemma}
\newtheorem{remark}[definition]{Remark}
\newtheorem{theorem}[definition]{Theorem}

\ifpdf
\hypersetup{
  pdftitle={Convergence of a Finite Volume Scheme for the Navier-Stokes-Korteweg Model via Dissipative Solutions},
  pdfauthor={J.\ Giesselmann, P.\ Öffner, and R. Sauerborn}
}
\fi

\begin{document}

\maketitle

\begin{abstract}

We propose a concept of dissipative weak (DW) solutions for the Navier–Stokes–Korteweg (NSK) system and prove conditional convergence of a structure-preserving finite volume scheme (FV)  towards such a solution. DW solutions provide a generalized solution concept in computational fluid dynamics and have recently attracted significant attention. They provide an extension of the famous Lax Equivalence Theorem to nonlinear problems, i.e. consistency and stability of a numerical scheme imply convergence.
Our work builds on recent advances where convergence towards DW solutions of structure-preserving schemes has been established for the Euler and Navier–Stokes equations. Indeed, we prove convergence of a recently proposed FV scheme by leveraging its conservation and  dissipation properties as well as its consistency.

\end{abstract}



\section{Introduction} 
\label{sec:introduction} 

We prove conditional convergence of a semi-discrete (in space) numerical scheme for the (local) isothermal Navier-Stokes-Korteweg (NSK) system
\begin{equation}\label{eq_Korteweg}
	\begin{aligned}
		\partial_t \rho+ \nabla \cdot  (\rho u) =&0  \\
		\partial_t (\rho \bu) + \nabla \cdot \left[  \rho \bu \otimes \bu +p(\rho) I  \right]  =&  \mu \Delta \bu + \kappa \nabla \cdot \left( \nabla \cdot (\rho \nabla\rho)\mathbf{I} - \frac12 |\nabla \rho|^2 \mathbf{I}\right)\\
		 & - \kappa \nabla \cdot \left( \nabla \rho \otimes \nabla \rho \right)
	\end{aligned}
\end{equation}
where the unknowns are the fluid density $\rho$ and the fluid velocity $\bu$. The pressure law $p=p(\rho)$ and the (constant) coefficients of viscosity, $\nu>0,$   and of capillarity, $\kappa >0$, are given.
For simplicity, as in \cite{GÖS}, we use $\nabla \bu$ as viscosity tensor.
The NSK system is a well known higher order system of PDEs that is frequently used to model multi-phase flows, thin film flows and can also be used for single phase flows, where, arguably, it captures more physical phenomena than the Navier-Stokes equations \cite{SlemrodReview}. We call \eqref{eq_Korteweg} the Euler-Korteweg (EK) system if $\mu=0$.\\
In terms of well-posedness of the local NSK system, we would like to mention existence of weak solutions if viscosity and capillarity are related in a specific way \cite{Antonelli2019GlobalEO} and well-posedness in Besov spaces for an NSK system with damping \cite{zbMATH07077618}. For the Euler-Korteweg-Poisson system non-uniqueness of weak entropy solutions was shown in \cite{WellposedKorteweg} using convex integration. This is expected to also hold for the EK system. Weak-strong uniqueness for weak entropy solutions of the EK system was shown, even for non-monotone pressure, in \cite{zbMATH06685093}. In this work, we consider the case of single phase flows, i.e.,  we focus on the case of a continuously differentiable and strictly
monotone pressure law $p$. This makes the first--order
part of the system strictly hyperbolic and ensures convexity of the associated
potential energy $P$, defined by $P'(\rho)\,\rho - P(\rho) = p(\rho).$
In particular, we will assume the polytropic pressure law 
$$p(\rho) = a \rho^\gamma \txt{with} a > 0, \gamma > 1 $$ holds. 

For the NSK system, a variety of numerical schemes have been presented in the past. Let us mention \cite{noblevilaEKthinfilmflow} which reduces the model order by introducing an auxiliary variable, \cite{dhaoadi11040876} using a hyperbolic relaxation, \cite{zbMATH06818400} using an elliptic relaxation,  \cite{zbMATH08071323} which is a local discontinuous Galerkin method, and \cite{GÖS} which is a finite volume scheme on Cartesian meshes.
We study the scheme derived  in \cite{GÖS}. It is particularly suitable for convergence analysis since it conserves mass and momentum exactly and dissipates energy. Moreover,  in contrast to most other schemes for NSK,
it discretises the equations directly without introducing auxiliary variables. Due to the strong nonlinearity of the equations a convergence analysis is very challenging for any of the schemes. Convergence results have only recently been obtained in the framework of energy-variational solutions in \cite{eiter2026}, and this, to our knowledge, is the first convergence analysis in the framework of dissipative solutions.

The approach we are taking is to prove a priori estimates on the numerical solutions in order to guarantee the existence of a weakly-$*$ convergent subsequence. Inspired by the works \cite{feireisl2021numerics}, we show convergence (along the above mentioned subsequence) of our numerical solutions towards dissipative solutions, provided the density of the numerical solutions is uniformly bounded from below, similar to the barotropic Euler case.
Indeed, based on the convergence properties of the scheme we will suggest a definition of dissipative solution for the NSK system. These can be interpreted as expected values of dissipative  measure valued solutions and satisfy the weak form of the equations up to defects.    We  will show that dissipative solutions of this type satisfy a dissipative-weak-strong uniqueness principle and, thus, constitute a meaningful generalization of standard solution concepts.
The remainder of this paper is structured as follows: In Section 2, we present the numerical scheme and recall its main properties from \cite{GÖS}. In Section 3, we define dissipative weak solutions and prove dissipative-weak-strong uniqueness of these solutions. Finally, in Section 4, we derive a priori estimates for numerical solutions, use them to pass to the limit and identify the limit as a dissipative solution.
 
\section{Semidiscrete Method}


In this section, we present the structure-preserving semi-discrete finite volume (FV) scheme introduced in \cite{GÖS},
whose convergence we aim to prove in this work.
We begin by specifying the notation and the general framework.

The computational domain is given by 
$
\mathbb{T}^2 = (\real / \mathbb{Z}) \times (\real / 
\mathbb{Z})
$
and divided into control volumes (cells) given by
$
\Omega_{i,j}
= (x_{i-\tfrac{1}{2}}, x_{i+\tfrac{1}{2}}]
\times
(y_{j-\tfrac{1}{2}}, y_{j+\tfrac{1}{2}}]
$,
with the cell interfaces defined by
$$
x_{i\pm \tfrac{1}{2}} = \left( i \pm \tfrac{1}{2}\right)  h_x, \quad
y_{j\pm \tfrac{1}{2}} = \left( j \pm \tfrac{1}{2}\right)  h_y
$$
where $h_x = M^{-1}, h_y = N^{-1}$ are the mesh sizes; for some $M,N \in \nat$.

%

For piecewise constant grid functions $\phi$ defined on this mesh, we introduce the following discrete operators:
\begin{align*}
	\Ld(\phi_{i,j}) 
	&\coloneqq \Delta_{h_x}\phi_{i,j} + \Delta_{h_y}\phi_{i,j} \\
	&\coloneqq \frac{\phi_{i+1,j}-2\phi_{i,j}+\phi_{i-1,j}}{h_x^2}
	+ \frac{\phi_{i,j+1}-2\phi_{i,j}+\phi_{i,j-1}}{h_y^2}, \\
	\nabla^c\phi_{i,j} 
	&\coloneqq		
	\begin{pmatrix}
		\cx\phi_{i,j} \\
		\cy\phi_{i,j}
	\end{pmatrix}
	\coloneqq
	\begin{pmatrix}
		\dfrac{\phi_{i+1,j} - \phi_{i-1,j}}{2h_x} \\
		\dfrac{\phi_{i,j+1} - \phi_{i,j-1}}{2h_y}
	\end{pmatrix}, \\
	\nabla^\pm\phi_{i,j} 
	&\coloneqq 
	\begin{pmatrix}
		\nabla^\pm_x\phi_{i,j} \\
		\nabla^\pm_y\phi_{i,j}
	\end{pmatrix}
	\coloneqq 
	\pm 
	\begin{pmatrix}
		\dfrac{\phi_{i\pm 1,j} - \phi_{i,j}}{h_x} \\
		\dfrac{\phi_{i,j \pm 1} - \phi_{i,j}}{h_y}
	\end{pmatrix}.
\end{align*}

Periodic boundary conditions on $\TT$ are imposed by identifying
\[
\phi_{i+kM,\, j+lN} \coloneqq \phi_{i,j}
\qquad \text{for all } k,l \in \mathbb{Z}.
\]
The initial data are defined via cell averages, see \eqref{eq:initCon}.

The semidiscrete FV scheme for the NSK system \eqref{eq_Korteweg}
in two spatial dimensions, written in terms of the conserved variables
$(\rho,\rho \yu) = (\rho,\rho u,\rho v)$, presented in \cite{GÖS}, is given by
\begin{subequations}\label{method:semi2d}
\begin{align}
- \frac{\mathrm d}{\mathrm dt}  \rho_{i,j}
&= \cx(\rho u)_{i,j}+ \cy(\rho v)_{i,j}
   - h\lambda \Ld\rho_{i,j},
\label{semi2drho}
\\
- \frac{\mathrm d}{\mathrm dt}  (\rho u)_{i,j}
&= \cx(\rho u^2)_{i,j}+ \cy(\rho u v)_{i,j}
   + \cx p(\rho_{i,j})
   - \mu \Ld(u_{i,j})
   - h\lambda \Ld(\rho u)_{i,j}
\nonumber\\
&\quad
- \kappa \Bigg[
\bx \frac{\rho_{i,j}\Ld\rho_{i+1,j}
      + \rho_{i+1,j}\Ld\rho_{i,j}}{2}
- \frac12 \bx \big(\fx\rho_{i,j}\big)^2
\label{semi2du}
\\
&\qquad
+ \frac12 \bx\!\left(
      \by\rho_{i+1,j}\by \rho_{i,j}\right)
- \by \!\left(
      \cx\rho_{i,j}\,\fy \rho_{i,j}\right)
\Bigg],
\nonumber\\
- \frac{\mathrm d}{\mathrm dt}  (\rho v)_{i,j}
&= \cy(\rho v^2)_{i,j}+ \cx(\rho u v)_{i,j}
   + \cy p(\rho_{i,j})
   - \mu \Ld(v_{i,j})
   - h\lambda \Ld(\rho v)_{i,j}
\nonumber\\
&\quad
- \kappa \Bigg[
\by \frac{\rho_{i,j}\Ld\rho_{i,j+1}
      + \rho_{i,j+1}\Ld\rho_{i,j}}{2}
- \frac12 \by \big(\fy\rho_{i,j}\big)^2
\label{semi2dv}
\\
&\qquad
+ \frac12 \by\!\left(
      \bx\rho_{i,j+1}\,\bx \rho_{i,j}\right)
- \bx \!\left(
      \cy\rho_{i,j}\, \fx \rho_{i,j}\right)
\Bigg].
\nonumber
\end{align}
\end{subequations}

In the artificial dissipation terms, we set $h := \max\{h_x,h_y\}$ and
\begin{equation}
	\label{eq:lambda}
	\lambda = \max \left\lbrace \demi \max_{i,j} \left\lbrace |\yu_{i,j}| + \sqrt{p'(\rho_{i,j})}\right\rbrace, \delta\right\rbrace, \txt{with $0 < \delta \ll 1$.}
\end{equation}
For the initial data we assume $(\rho_0,\mathbf u_0): \mathbb{T}^2 \to \real \times \real^2$ are measurable
with
$\rho_0 \ge \underline{\rho} > 0$ and finite total energy,
\[
 \Ensk(\rho_0,\mathbf u_0) < \infty
\]
where
\begin{equation}\label{def:energyrhom}
 \Ensk(\rho,\mathbf u) \coloneqq  \int_\TT \ENSK(\rho,\mathbf u) \di x, \qquad \ENSK(\rho,\mathbf u) \coloneqq P(\rho) + \frac12 \rho |\yu|^2 + \frac{\kappa}{2} |\nabla \rho|^2.
\end{equation}
A discrete counterpart of the energy functional $\mathcal{E}$ is defined by
\begin{align}\label{eq:energy2d}
\Ensk^h(\rho, \yu)
=
\frac{1}{MN}\sum_{i=1}^{M}\sum_{j=1}^{N}
\left(
\frac{1}{2} \rho_{i,j} |\mathbf u_{i,j}|^2
+ P(\rho_{i,j})
+ \frac{\kappa}{2} \left|\nabla_h^+ \rho_{i,j}\right|^2
\right),
\end{align}

We further introduce the piecewise constant reconstructions
$(\rho^h,\mathbf u^h)$  defined by
\[
\rho^h(t,x)
=\sum_{i=1}^{M}\sum_{j=1}^{N}
\rho_{i,j}(t)\mathbf 1_{\Omega_{i,j}}(x),
\quad
\mathbf{u}^h(t,x)
=\sum_{i=1}^{M}\sum_{j=1}^{N}
\mathbf{u}_{i,j}(t)\mathbf 1_{\Omega_{i,j}}(x).
\]

For functions $f:\mathbb{T}^2\to\mathbb R$ we define the discrete
forward, backward and central difference operators by
\[
\begin{aligned}
D_{h,x}^{\pm} f(x,y)
&=
\pm\frac{f(x\pm h_x,y)-f(x,y)}{h_x},
&
D_{h,x}^{c} f(x,y)
&=
\frac{f(x+h_x,y)-f(x-h_x,y)}{2h_x},
\\
D_{h,y}^{\pm} f(x,y)
&=
\pm\frac{f(x,y\pm h_y)-f(x,y)}{h_y},
&
D_{h,y}^{c} f(x,y)
&=
\frac{f(x,y+h_y)-f(x,y-h_y)}{2h_y}.
\end{aligned}
\]

Moreover, we define
\[
D_h^{\pm} f
\coloneqq
\big(D_{h,x}^{\pm}f, D_{h,y}^{\pm}f\big)^T,
\qquad
D_h^{c} f
\coloneqq
\big(D_{h,x}^{c}f, D_{h,y}^{c}f\big)^T,
\]
and
\[
D_h^2 f
\coloneqq
D_{h,x}^{-}D_{h,x}^{+}f
+
D_{h,y}^{-}D_{h,y}^{+}f.
\]
We will also use $D^\bullet_{h}$ in equations that hold for forward, backward, and central differences alike.

The reconstructed functions $(\rho^h,\mathbf u^h)$ then satisfy
\begin{equation}\
\begin{aligned}
0=&
\partial_t\rho^h
+ D_{h,x}^{c}(\rho^h u^h)
+ D_{h,y}^{c}(\rho^h v^h)
- \lambda h D_h^2\rho^h,
\\0=&
\partial_t(\rho^h u^h)
+ D_{h,x}^{c}(\rho^h (u^h)^2)
+ D_{h,y}^{c}(\rho^h u^h v^h)
+ D_{h,x}^{c} p(\rho^h)
\\
&\,
- \lambda h D_h^2(\rho^h u^h)
- \mu D_h^2 u^h \notag
\\
&\, - \kappa\Bigg[ \Dhmx\frac{\rho^h D_h^2\rho^h(\cdot+h_x\be_x)+\rho^h(\cdot+h_x\be_x) D_h^2\rho^h}{2} - \demi \Dhmx(\Dhpx\rho^h)^2 \notag
\\
&\qquad
+\demi \Dhmx\big(  \Dhmy\rho^h\Dhmy \rho^h(\cdot+h_x\be_x)\big)   - \Dhmy \big( \Dhcx\rho^h\, \Dhpy \rho^h\big)\Bigg],
\\0=&
\partial_t(\rho^h v^h)
+ D_{h,y}^{c}(\rho^h (v^h)^2)
+ D_{h,x}^{c}(\rho^h u^h v^h)
+ D_{h,y}^{c} p(\rho^h)
\\
&\,
- \lambda h D_h^2(\rho^h v^h)
- \mu D_h^2 v^h \notag
\\
&\, - \kappa\Bigg[ \Dhmy\frac{\rho^h D_h^2\rho^h(\cdot+h_y\be_y)+\rho^h(\cdot+h_y\be_y) D_h^2\rho^h}{2} - \demi \Dhmy(\Dhpy\rho^h)^2 \notag
\\
&\qquad
+\demi \Dhmy\big(  \Dhmx\rho^h\Dhmx \rho^h(\cdot+h_y\be_y)\big)   - \Dhmx \big( \Dhcy\rho^h\, \Dhpx \rho^h\big)\Bigg]
\end{aligned}
\end{equation}
in $(0,T)\times\mathbb{T}^2$,
with initial condition
$
(\rho^h,\mathbf m^h)(0)
=
(\rho_0^h,\mathbf m_0^h).
$
The discrete initial data are defined cellwise by
\begin{align}
	\label{eq:initCon}
	\rho_0^h(x)
	=
	\sum_{i=1}^{M}\sum_{j=1}^{N}
	\mathbf 1_{\Omega_{i,j}}(x)
	\frac{1}{h_x h_y}
	\int_{\Omega_{i,j}}\rho_0(y)\,\mathrm dy,
	\\
	\yu_0^h(x)
	=
	\sum_{i=1}^{M}\sum_{j=1}^{N}
	\mathbf 1_{\Omega_{i,j}}(x)
	\frac{1}{h_x h_y}
	\int_{\Omega_{i,j}}\yu_0(y)\,\mathrm dy.\notag
\end{align}
Finally, the discrete energy \eqref{eq:energy2d} admits the integral representation
\begin{equation}\label{eq:energy.discrete}
\Ensk^h(\rho^h, \yu^h)
=
\int_\TT
\left(
\frac12 \rho^h |\mathbf u^h|^2
+ P(\rho^h)
+ \frac{\kappa}{2}
|D_h^+\rho^h|^2
\right)
\di x.
\end{equation}
As demonstrated in \cite{GÖS}, the scheme \eqref{method:semi2d} has the following properties: 
\begin{itemize}
\item It conserves total mass and momentum.
\item It dissipates energy; in particular, the inequality
\begin{equation}\label{eq:energydissipation}
	\frac{\di}{\di t}\Ensk[\rho^h, \yu^h] \leq  - \mu \int_\TT \left|D^+_h \yu\right|^2 \di x - \kappa \int_\TT \lambda h (D_h^2\rho^h)^2 \di x \leq 0
\end{equation}holds.

\end{itemize}

\section{Dissipative Solutions to the NSK system}

There exist several 
solution concepts for the Euler and Navier–Stokes equations, such as strong solutions, weak entropy solutions, and others; see \cite{feireisl2021numerics} and the references therein. The concept of  (dissipative) measure-valued solutions dates back to the 1980s \cite{DiPerna}  and has later been further extended, cf. \cite{MeasCompNavSto}.
Recently, such solutions have also been defined for Korteweg-type systems \cite{yang2026dissipative, zbMATH07798933}, and their existence has been established.
In this work, we focus on dissipative weak (DW) solutions, a refined concept of measure-valued solutions. These can be seen as the barycenters of dissipative  measure-valued solutions \cite{feireisl2021numerics}
and satisfy the weak formulation up to a defect measure. We define DW solutions for the NSK system as follows:
\begin{definition}[DW solution for the NSK equations]\label{def_dmv}
We say that $(\rho,\bu)$ is a 
\emph{dissipative weak solution} to the NSK system \eqref{eq_Korteweg} with periodic boundary conditions if the following conditions hold:\\
\begin{itemize}

\item 
$
\rho \in L^{\infty}(0,T;H^1(\mathbb{T}^d)) 
\cap C_{\mathrm{weak}}([0,T];L^{\gamma}(\mathbb{T}^d)),
$
\item 
$
\bu \in \begin{cases} 
	L^2(0,T; L^2(\mathbb{T}^d)) &\txt{for $\mu = 0$}\\
	L^2(0,T; H^1(\mathbb{T}^d))  &\txt{for $\mu > 0,$}
\end{cases} 
$
\item
$
\mathbf{m}=\rho \bu \in 
C_{\mathrm{weak}}\!\left([0,T];
L^{\frac{2\gamma}{\gamma+1}}(\mathbb{T}^d)\right),
$

    \item The integral identity corresponding to the continuity equation
    \begin{equation}\label{eq:dw_continuity}
        \left[
        \int_{\mathbb{T}^d} \varphi\, \rho \,\mathrm{d}x
        \right]_{t=0}^{t=\tau}
        =
        \int_0^{\tau}\!\!\int_{\mathbb{T}^d}
        \left(
        \partial_t \varphi\, \rho
        +
        \rho \bu \cdot \nabla \varphi
        \right)
        \,\mathrm{d}x\,\mathrm{d}t
    \end{equation}
    holds for any $0 \leq \tau \leq T$ and any test function 
    $\varphi \in W^{1,\infty}((0,T)\times \mathbb{T}^d)$.
\end{itemize}
\begin{itemize}
    \item The integral identity corresponding to the momentum equation
    \begin{equation}\label{eq:dw_momentum}
        \begin{aligned}
\Bigg[
\int_{\mathbb{T}^d} 
\mathbf{m} \cdot \boldsymbol{\varphi} 
\, \mathrm{d}x
\Bigg]_{t=0}^{t=\tau}
&=
\int_{0}^{\tau} \!\! \int_{\mathbb{T}^d}
\Big[
\mathbf{m} \cdot \partial_t \boldsymbol{\varphi}
+ \rho \mathbf{u} \otimes \mathbf{u} : \nabla \boldsymbol{\varphi}
+ p(\rho)\, \operatorname{div} \boldsymbol{\varphi}
\Big]
\, \mathrm{d}x \, \mathrm{d}t
\\
&\quad
- \mu 
\int_{0}^{\tau} \!\! \int_{\mathbb{T}^d}
\nabla \mathbf{u} : \nabla \boldsymbol{\varphi}
\, \mathrm{d}x \, \mathrm{d}t
\\
&\quad
+ \kappa 
\int_{0}^{\tau} \!\! \int_{\mathbb{T}^d}
\left(
\nabla \rho \otimes \nabla \rho
+\frac{1}{2} |\nabla \rho|^2 I
\right)
: \nabla \boldsymbol{\varphi}
\, \mathrm{d}x \, \mathrm{d}t\\
& \quad + \kappa
\int_{0}^{\tau} \!\! \int_{\mathbb{T}^d}
\rho \nabla \rho \cdot \nabla \operatorname{div}  \boldsymbol{\varphi} 
\, \mathrm{d}x \, \mathrm{d}t
\\
&\quad
+ \int_{0}^{\tau} \!\! \int_{\mathbb{T}^d}
\nabla \boldsymbol{\varphi}
: \mathrm{d}\mathfrak{R}
        \end{aligned}
    \end{equation}
    holds for any $0 \leq \tau \leq T$ and any test function
    $\boldsymbol{\varphi} \in C^\infty([0,T]\times \mathbb{T}^d;\mathbb{R}^d)$,
    with some defect measure
    \[
    \mathfrak{R} \in 
    L^{\infty}(0,T;
    \mathcal{M}^+(\mathbb{T}^d;\mathbb{R}^{d\times d}_{\text{sym}}))
    \]
    where $\mathcal{M}^+(\mathbb{T}^d;\mathbb{R}^{d\times d}_{\text{sym}})$ denotes  the set of positive
semi-definite matrix-valued measures.
    \item (Energy inequality) There exists a defect measure
    \[
    \mathfrak{E} \in 
    L^{\infty}(0,T;
    \mathcal{M}^+(\mathbb{T}^d))
    \]
    such that the energy inequality
    \begin{align}\label{eq:dw_energy}
        \int_{\mathbb{T}^d}
       \ENSK(\rho,\bu)(\tau)
        \di x
        + \mu
        \int_{0}^{\tau}\!\!\int_{\mathbb{T}^d}
        |\nabla \bu|^2
        &\di x \di t
        + \int_{\mathbb{T}^d}
        \mathrm{d}\mathfrak{E}(\tau)
        \\
        &\leq
        \int_{\mathbb{T}^d}
        \ENSK(\rho_0,\bu_0)
        \di x \notag
    \end{align}
    holds for almost all $0 \leq \tau \leq T$.
        \item (Defect compatibility condition) There exist constants $0 < c \leq C$ such that
    \[
    c\,\mathfrak{E}
    \leq
    \operatorname{tr}[\mathfrak{R}]
    \leq
    C\,\mathfrak{E}.
    \]
    holds.
\end{itemize}
\end{definition}

\begin{remark}
Here, we have defined DW solutions with respect to the quantities  $(\rho,\bu)$  similar to Navier-Stokes. For the Euler-Korteweg case, we could replace the velocity by the momentum  in the definition which would simplify the convergence analysis of the numerical scheme similar to Euler and Navier-Stokes cases from  \cite{feireisl2021numerics}.
\end{remark}
Next, we will investigate if a weak-strong uniqueness result holds. Therefore, we introduce the relative energy.
%

The relative energy between two pairs of functions $(\rho,\bu)$ and $(\tilde\rho,\tilde \bu)$ is defined by
\begin{equation}
\begin{aligned}
\mathbb{E}(\rho,\bu \mid \tilde\rho,\tilde \bu)
&=
P(\rho)
+
\frac12 \rho |\bu|^2
+
\frac{\kappa}{2} |\nabla \rho|^2
-
\left(
P(\tilde\rho)
+
\frac12 \tilde\rho |\tilde \bu|^2
+
\frac{\kappa}{2} |\nabla \tilde\rho|^2
\right)
\nonumber
\\
&\quad
-
\left(
P'(\tilde\rho)
-
\frac12 |\tilde u|^2
-
\kappa \Delta \tilde\rho
\right)
(\rho-\tilde\rho)
-
\tilde \bu \cdot (\rho \bu - \tilde\rho \tilde \bu).
\end{aligned}
\end{equation}

Let $(\rho,\bu)$ be a dissipative weak solution of \eqref{eq_Korteweg} and let $(\tilde\rho,\tilde \bu)$ be a smooth solution of \eqref{eq_Korteweg}, i.e. we assume for its regularity:
\[
\tilde{\bu} \in W^{2,\infty}((0,T)\times\mathbb T^d),
\qquad
\tilde\rho \in W^{3,\infty}((0,T)\times\mathbb T^d).
\]
Then, we obtain by regrouping terms and using the smooth solution or functions of it as test functions for \eqref{eq:dw_continuity} and \eqref{eq:dw_momentum}, for a.a. $\tau \in (0,T)$:
\begin{equation}\label{REI-NSK-2D}
\begin{aligned}
&\frac{\di}{\di t}
\int_{\mathbb T^d}
\mathbb{E}(\rho,\rho \bu \mid \tilde\rho,\tilde\rho \tilde{\bu})
\, \mathrm{d}x
+
\mu
\int_{\mathbb T^d}
|\nabla (\bu-\tilde{\bu})|^2
\, \mathrm{d}x
\\
&\le
\int_{\mathbb T^d}
\mathcal T(\rho,\bu \mid \tilde\rho,\tilde{\bu})
\, \mathrm{d}x
\\
&\quad
-\int_{\mathbb T^d} \kappa \Bigg[ \nabla \operatorname{div}\yuti \cdot (\rho - \rhoti)\left( \nabla \rho - \nabla \rhoti\right)
+
\demi \operatorname{div} \yuti \left| \nabla \rho - \nabla \rhoti\right|^2
\\
&
\qquad\qquad +
\nabla \tilde{\bu} :
\Big(
(\nabla\rho-\nabla\tilde\rho)
\otimes
(\nabla\rho-\nabla\tilde\rho)
\Big)\Bigg] 
\, \mathrm{d} x
\\
&\quad
+
C (\|\tilde{\bu}\|_{W^{2,\infty}(\mathbb T^d)})
\int_{\mathbb T^d}
\mathrm{d}\mathfrak E(t).
\end{aligned}
\end{equation}

\noindent Here $\mathcal T(\rho,\bu \mid \tilde\rho,\tilde{\bu})$ denotes the standard transport and pressure remainder terms (as in the viscous compressible Navier--Stokes/Euler system without capillarity). We refer to Chapter 6.3 in \cite{feireisl2021numerics} for details on
how to handle these terms and to \cite{zbMATH06685093} for details on the derivation of the capillarity terms.

%
%
\noindent Using the regularity of the strong solution, all remainder terms in  \eqref{REI-NSK-2D} can be estimated by the relative energy.
We obtain the relative energy inequality 
\begin{align}
\frac{\di}{\di t}
\mathcal E_{\mathrm{rel}}(t)
+
\mu
\|\nabla(\bu-\tilde{\bu})\|_{L^2(\mathbb T^d)}^2
&\le
C\!\left(
\|\tilde{\bu}\|_{W^{2,\infty}}
\right) \left( 
\mathcal E_{\mathrm{rel}}(t)
\nonumber
+
\int_{\mathbb T^d} \mathrm{d}\mathfrak E(t)\right) ,
\end{align}
where
\[
\mathcal E_{\mathrm{rel}}(t)
=
\int_{\mathbb T^d}
\mathbb{E} (\rho,\rho \bu \mid \tilde\rho,\tilde\rho\tilde{\bu})
\, \mathrm{d}x .
\]

Integrating the differential inequality in time gives
\begin{multline}\label{WSU-NSK-int}
\mathcal E_{\mathrm{rel}}(\tau)
+
\mu \int_0^\tau \|\nabla(\bu-\tilde{\bu})\|_{L^2(\mathbb T^d)}^2 \, \mathrm{d}t
+
\int_{\mathbb T^d} \mathrm{d}\mathfrak E(\tau)\\
\le
\mathcal E_{\mathrm{rel}}(0)
+
\int_{\mathbb T^d} \mathrm{d}\mathfrak E(0)
+
C \int_0^\tau \Big( \mathcal E_{\mathrm{rel}}(t) + \int_{\mathbb T^d} \mathrm{d}\mathfrak E(t) \Big) \di t
\end{multline}
and applying Gr\"onwall's lemma to \eqref{WSU-NSK-int} yields
\begin{equation}
\begin{aligned}
\mathcal E_{\mathrm{rel}}(t)
+
\mu \int_0^t &\|\nabla(\bu-\tilde{\bu})\|_{L^2(\mathbb T^d)}^2 \, \mathrm{d}t
+
\int_{\mathbb T^d} \mathrm{d}\mathfrak E(t)
\\
\le&
\left(
\mathcal E_{\mathrm{rel}}(0)
+
\int_{\mathbb T^d} \mathrm{d}\mathfrak E(0)
\right)
\exp(Ct),
\end{aligned}
\end{equation}
for all $t\in[0,T]$. \\
In particular, if the initial data coincide, then
$
\mathcal E_{\mathrm{rel}}(0)=0, \;
\int_{\mathbb T^d} d\mathfrak E(0)=0,
$
and consequently
$
\mathcal E_{\mathrm{rel}}(t)=0, 
\int_{\mathbb T^d} d\mathfrak E(t)=0
\quad \text{for all } t\in[0,T].
$\\

\noindent Since the relative energy functional satisfies\\
$
\mathcal E_{\mathrm{rel}}(t)=0
 \Longleftrightarrow 
\rho=\tilde\rho
\text{ and }
\bu=\tilde{\bu}
\text{ a.e. in } \mathbb T^d,
$
we conclude
\[
(\rho,\bu)=(\tilde\rho,\tilde{\bu})
\quad \text{a.e. in } (0,T)\times\mathbb T^d.
\]

Hence, the strong solution is unique within the class of DW solutions  emanating from the same initial data. In other words, the NSK system with constant viscosity  fulfills the weak--strong uniqueness property.

We can summarize our result in  the following theorem: 
\begin{theorem}[Dissipative-Weak-Strong Uniqueness]
Let $T>0$. Let $(\ro, \yu)$ be a dissipative weak (DW) solution of the NSK system \eqref{eq_Korteweg} with initial data $(\ro_0, \yu_0)$ and periodic boundary conditions on $(0,T) \times \mathbb{T}^d$. Let $(\roti, \yuti)$ be a strong solution on $(0,T) \times \mathbb{T}^d$ to the same system with the same initial data. Then
\[
(\ro, \yu) = (\roti, \yuti) \quad \text{almost everywhere in } (0,T)\times \mathbb{T}^d.
\]
\end{theorem}

\section{Convergence of a structure-preserving FV scheme via DW solutions}

In this section, 
we demonstrate that the proposed FV scheme  \eqref{semi2drho}--\eqref{semi2dv} converges in the framework of DW solutions for the NSK equations  \eqref{eq_Korteweg}. To the best of our knowledge this is the first convergence proof in that framework and on this level of generality for a practical numerical method for the NSK model, i.e.,  a scheme which is not constructed solely for analytical reasons.
Similar to the Euler or Navier-Stokes case, described in detail in   \cite{feireisl2021numerics}, we need some assumptions to pass to the weak limit.
Let $(\rho^h, \mathbf u^h)$ be the approximate solutions generated by the semi-discrete FV scheme \eqref{method:semi2d} on Cartesian grids with mesh size $h_n \to 0$. Then, we assume the following:\\
\begin{assumption}~\\[-0.3cm]\label{assump}
\begin{itemize}
%
%
	\item \textbf{Uniform existence:}
	There exists a time interval $[0,T]$ with $T>0$ independent of $h$ such that \eqref{semi2drho}--\eqref{semi2dv} admits a solution on $[0,T]$.
	\item \textbf{Positivity of density:} 
	There exists a positive constant $\underline{\rho}$ such that
	$0 < \underline{\rho} \le \rho^h$ for all $t \in [0,T]$ and  for $h \to 0$ uniformly. 
\end{itemize}
\end{assumption}
%

The following estimates then follow directly.

\begin{lemma}[Uniform bounds from discrete energy]\label{lem:energy-bounds}
Under assumption \ref{assump} the following  sequences are bounded uniformly-in-$h$:
\begin{align*}
\mathbf u^h &\in L^\infty(0,T;L^2(\TT)),\\
\rho^h &\in L^\infty(0,T;L^\gamma(\TT)),\\
\nabla^c \rho^h &\in L^\infty(0,T;L^2(\TT)).
\end{align*}
\end{lemma}

\begin{proof}
 We obtain directly from the kinetic term 
$
\frac12 \int_{\TT} \rho^h |\mathbf u^h|^2 \mathrm{d}x \le \Ensk^h(t) \le C.
$
Using the lower bound $\rho^h \ge \underline{\rho}$, we get
\[
\int_{\TT} |\mathbf u^h|^2 \mathrm{d}x \le \frac{2}{\underline{\rho}} \Ensk^h(t) \le \frac{2C}{\underline{\rho}}.
\]
Hence $\mathbf u^h \in L^\infty(0,T;L^2(\TT))$ with uniform bound.\\
Due to the convexity of the pressure potential, the following bound on the density holds:
\[
\int_{\TT} P(\rho^h) \mathrm{d}x \le \Ensk^h(t) \le C \quad \implies \quad
\int_{\TT} (\rho^{h})^{\gamma} \mathrm{d}x \le C',
\]
i.e., $\rho^h \in L^\infty(0,T;L^\gamma(\TT))$ is uniformly bounded.
Finally, the capillarity term gives directly
\[
\frac{\kappa}{2} \int_{\TT} |\nabla^c \rho^h|^2 \mathrm{d}x \le \Ensk^h(t) \le C,
\]
hence
\[
\|\nabla^c \rho^h\|_{L\infty(0,T;L^2(\TT))}^2 \le \frac{2C}{\kappa}.
\]
\end{proof}
\begin{remark}
Instead of working with central differences inside the discrete gradient estimate in Lemma \ref{lem:energy-bounds}, similar results hold for forward and backward differences. Additionally, a uniform bound for $\rho^h$ in $L^\infty(0,T;L^2(\TT))$ can be derived using a discrete version of the Poincaré-Friedrichs inequality, and $L^\infty L^p$ bounds for $1 \leq p < \infty$ are available under some additional assumptions, cf. \eqref{eq:conv.rho.s}.
\end{remark}

The basic idea of the convergence proof is that we demonstrate an analogous result to Lax-Richtmyer equivalence, which, roughly speaking, states that consistency and stability imply convergence. Therefore, we will investigate the consistency error of the scheme, similar to analyses of the Euler equations in \cite{abgrall2022convergence, lukacova2022convergence,zbMATH08060826} and Navier-Stokes equations in \cite{zbMATH07528308,zbMATH07559967,zbMATH08106216}.
Here, consistency means that when inserting the numerical approximations into the weak formulation, the consistency error  tends  to zero under grid refinement. We prove the following result:

\begin{theorem}[Consistency of the FV scheme to the NSK system]
Let $(\rho^h,\bu^h)$ be the piecewise constant reconstruction of a solution of the semi-discrete FV-scheme 
\eqref{method:semi2d} with the initial condition $(\rho_0,\bu_0)	\colon \TT \to \real^+ \times \real^2$ measurable with $\mathcal E_{\text{NSK}}(\rho_0,\bu_0) < \infty$, and let Assumption \ref{assump} hold.


Then the approximate
solutions satisfy the weak formulation of the NSK system \eqref{eq_Korteweg} up to
consistency error terms. Indeed,
\begin{equation}\label{NSK-consistency-mass}
\begin{aligned}
-\int_{\TT} \rho^h(0,x)\,\varphi(0,x)\di x
&=
\int_0^T \int_{\TT}
\left(
\rho^h \partial_t \varphi
+
\rho^h \bu^h \cdot \nabla \varphi
\right)
\di x \di t
+
\int_0^T R_1(t,\varphi)\di t,
\end{aligned}
\end{equation}
%
holds for any $\varphi \in C_c^\infty((0,T)\times\TT)$, and
\begin{equation}\label{NSK-consistency-momentum}
\begin{aligned}
&-\int_{\TT} (\rho^h\bu_h)(0,x)\cdot \boldsymbol{\varphi}(0,x)\di x\\
=&
\int_0^T \int_{\TT}
\Big(
\rho^h \bu_h \cdot \partial_t \boldsymbol{\varphi}
+
(\rho^h \bu_h \otimes \bu_h + p(\rho^h)\, \mathbf I)
: \nabla \boldsymbol{\varphi}
\Big)
\di x \di t
\\
&
- \mu \int_0^T \int_{\TT}
D^c_h \bu_h : \nabla \boldsymbol{\varphi}\di x \di t
\\
&
+ \kappa \int_0^T \int_{\TT}
\left(
D^c_h\rho^h \otimes D^c_h\rho^h
+ \frac12 |D^c_h\rho^h|^2 \, \mathbf I
\right)
: \nabla\boldsymbol{\varphi}
\di x \di t
\\
&
+  \kappa \int_0^T \int_{\TT}
\rho^h D^c_h\rho^h 
\cdot \nabla \operatorname{div} \boldsymbol{\varphi}
\di x \di t
\\
&
+
\int_0^T R_2(t,\boldsymbol{\varphi})\di t ,
\end{aligned}
\end{equation}
holds for any $\boldsymbol{\varphi}\in C_c^\infty((0,T)\times\TT;\mathbb R^2)$.\\

\noindent Moreover, the consistency error terms satisfy
\[
\|R_1(t,\varphi)\|_{L^1(0,T)}
\to 0, \quad
\|R_2(t,\boldsymbol{\varphi})\|_{L^1(0,T)}
\to 0 
\txt{as $h \to 0$.}
\]

\end{theorem}
%

\begin{proof}
Let $g_h: [0,T] \times \TT \rightarrow \mathbb{R}$ be a function so that $g_h(t,\cdot)$ is cell-wise constant for all $t\in [0,T]$ and that is continuously differentiable in time.
First, we observe that for all 
\[
\phi \in C^{\infty}([0,T]\times \TT)
\]
the following identity holds:
\begin{equation}\label{time-id}
\left.
\int_{\TT} g_h \phi \di x
\right|_{t=0}^{t=\tau}
=
\int_0^\tau
\int_{\TT}
\frac{\di}{\di t}(g_h \phi)\di x \di t
=
\int_0^\tau
\int_{\TT}
\left(
g_h \partial_t \phi
+
\phi\,\partial_t g_h
\right) \di x \di t .
\end{equation}

We denote the
projection of $\phi$  onto the function given by its cell averages by $\phi_h$. Then, we have
\begin{equation}\label{interp-split}
\int_0^\tau
\int_{\TT}
\phi\,\partial_t g_h\di x \di t
=
\int_0^\tau
\int_{\TT}
(\phi-\phi_h)\,\partial_t g_h\\di x \di t
+
\int_0^\tau
\int_{\TT}
\phi_h\,\partial_t g_h\di x \di t .
\end{equation}

The first term on the right-hand side vanishes as, for each $t\in [0,T]$), $(\phi(t,\cdot)-\phi_h(t,\cdot))$ is orthogonal to all cell-wise constant functions such as $\partial_t g_h(t,\cdot)$.
Therefore, it suffices to consider the second term in \eqref{interp-split}. In particular, we focus on the momentum equation \eqref{NSK-consistency-momentum}, as the density equation \eqref{NSK-consistency-mass} can be handled analogously to the arguments in \cite{feireisl2021numerics}.
We therefore multiply the momentum equation from the FV scheme with corresponding cell averages of a test function and sum over all cells, which yields, after discrete integration by parts:
\begin{equation}\label{eq_semi}
\begin{aligned}
&\int_0^\tau \!\!
\int_{\TT}
\left(
\boldsymbol{\varphi_h} \cdot \partial_t\left( \rho^h \yu_h\right)  \
\right) \mathrm{d}x \, \mathrm{d}t \\
=&\int_0^\tau\!\!
\int_{\TT}
(\rho^h \bu^h\otimes\bu^h+p(\rho^h)I):D_h^c \boldsymbol{\varphi}_h \dd x\dd t\\ 
    &
    + \int_0^T\!\!\int_{\TT} -\lambda h  \rho^h \bu^h\cdot D^2_h\varphi 
    -\mu D_h^+ \bu^h:D_h^+ \boldsymbol{\varphi_h} \dd x \dd t
    \\
    &
    - \int_0^T\!\!\int_{\TT} \kappa \Bigg[ 
    \frac{\rho^h D^2_h\rho^h(\cdot+h_x\be_x)+\rho^h(\cdot+h_x\be_x) D^2_h\rho^h}{2}\Dhpx\varphi_x
       \\
    &\qquad
    +\frac{ \rho^h D^2_h\rho^h(\cdot+h_y\be_y)+\rho^h(\cdot+h_y\be_y) D^2_h\rho^h}{2}\Dhpy\varphi_y
     \\
    &\qquad
    - \demi (\Dhpy\rho^h)^2\Dhpy\varphi_y  - \demi (\Dhpx\rho^h)^2\Dhpx\varphi_x
    \\
    &\qquad
    +\demi\Dhmy\rho^h\Dhmy \rho^h(\cdot+h_x\be_x)\Dhpx\varphi_x
    \\
    &\qquad
    +\demi\Dhmx\rho^h\Dhmx \rho^h(\cdot+h_y\be_y)\Dhpy\varphi_y
    \\
    &\qquad
    -\Dhcx\rho^h\, \Dhpy \rho^h\Dhpy\varphi_x
    -\Dhcy\rho^h\, \Dhpx \rho^h\Dhpx\varphi_y
    \Bigg] \dd x \dd t
\end{aligned}
\end{equation}
with $\boldsymbol{\varphi_h}=(\varphi_x, \varphi_y)$. The Navier–Stokes terms can be derived and estimated similarly to the analysis in \cite{feireisl2021numerics}. 

The capillarity part in \eqref{eq_semi} is what differs from \cite{feireisl2021numerics} and needs to be investigated in order
 to derive \eqref{NSK-consistency-momentum}.
We have the original terms: 
\begin{equation}\label{K_Terms_Conv_Orig}
\begin{aligned}
- \kappa &\int_{0}^T  \!\! \int_{\TT}
    \frac{\rho^h D^2_h\rho^h(\cdot+h_x\be_x)+\rho^h(\cdot+h_x\be_x) D^2_h\rho^h}{2}\Dhpx\varphi_x
       \\
    &
    +\frac{ \rho^h D^2_h\rho^h(\cdot+h_y\be_y)+\rho^h(\cdot+h_y\be_y) D^2_h\rho^h}{2}\Dhpy\varphi_y
    \\
    &
    - \demi (\Dhpy\rho^h)^2\Dhpy\varphi_y - \demi (\Dhpx\rho^h)^2\Dhpx\varphi_x
    \\
    &
   { +\demi\Dhmy\rho^h\Dhmy \rho^h(\cdot+h_x\be_x)\Dhpx\varphi_x}
   {+\demi\Dhmx\rho^h\Dhmx \rho^h(\cdot+h_y\be_y)\Dhpy\varphi_y}
    \\
    &
    -\Dhcx\rho^h\, \Dhpy \rho^h\Dhpy\varphi_x
    -\Dhcy\rho^h\, \Dhpx \rho^h\Dhpx\varphi_y
     \di x \di t. 
\end{aligned}
\end{equation}
In the next step, we apply discrete integration by parts in the
$x$-direction to the remaining $D^2_h \rho^h$ terms and combine this with the discrete product rule to obtain
\begin{equation*}
\begin{aligned}
\int_{\TT}& \frac{\rho^h D^2_h\rho^h(\cdot+h_x\be_x)+\rho^h(\cdot+h_x\be_x) D^2_h\rho^h}{2} \Dhpx\varphi_x\dd x
\\
&\quad
=-\frac{1}{2}\int_{\TT}\Dhmx\rho^h\Dhmx\rho^h(\cdot+h_x\be_x)\Dhpx\varphi_x(\cdot-h_x\be_x)
\\
&\qquad\qquad+\rho^h\Dhmx\rho^h(\cdot+h_x\be_x)\Dhmx\Dhpx\varphi_x
\\
&\qquad\qquad
+\Dhmy\rho^h\Dhmy\rho^h(\cdot+h_x\be_x)\Dhpx\varphi_x(\cdot-h_y\be_y)
\\
&\qquad\qquad
+\rho^h\Dhmy\rho^h(\cdot+h_x\be_x)\Dhmy\Dhpx\varphi_x
\\
&\qquad\qquad
+\Dhmx\rho^h(\cdot+h_x\be_x) \Dhmx\rho^h \Dhpx\varphi_x
\\
&\qquad\qquad+\rho^h\Dhmx\rho^h\Dhmx\Dhpx\varphi_x
\\
&\qquad\qquad
+
\Dhmy\rho^h(\cdot+h_x\be_x) \Dhmy\rho^h \Dhpx\varphi_x(\cdot-h_y\be_y)
\\
&\qquad\qquad
+ \rho^h(\cdot+h_x\be_x)\Dhmy\rho^h\Dhmy\Dhpx\varphi_x\dd x
\\
&\qquad
=-\int_{\TT} \Dhpx\rho^h \Dhmx\rho^h \Dhcx\varphi_x
+ \rho^h\Dhcx\rho^h\Dhmx\Dhpx\varphi_x
\\
&\qquad\qquad
+\Dhmy\rho^h \Dhmy\rho^h(\cdot+h_x\be_x)\Dhpx\varphi_x(\cdot-h_y\be_y)\dd x
\\
&\qquad\qquad
+ \frac{\rho^h\Dhmy\rho^h(\cdot+h_x\be_x)+\rho^h(\cdot+h_x\be_x)\Dhmy\rho^h}{2}\Dhmy\Dhpx\varphi_x \dd x,
\end{aligned}
\end{equation*}
and similarly in the $y$-direction:
\begin{equation*}
\begin{aligned}
&\int_{\TT} \frac{\rho^h D^2_h\rho^h(\cdot+h_y\be_y)+\rho^h(\cdot+h_y\be_y) D^2_h\rho^h}{2} \Dhpy\varphi_y\dd x
\\
=&-\int_{\TT} \Dhpy\rho^h \Dhmy\rho^h \Dhcy\varphi_y
+ \rho^h\Dhcy\rho^h\Dhmy\Dhpy\varphi_y
\\ 
&\quad +\Dhmx\rho^h \Dhmx\rho^h(\cdot+h_y\be_y) \Dhpy\varphi_y(\cdot-h_x\be_x)\dd x
\\
&
\quad + \frac{\rho^h\Dhmx\rho^h(\cdot+h_y\be_y)+\rho^h(\cdot+h_y\be_y)\Dhmx\rho^h}{2}\Dhmx\Dhpy\varphi_y \dd x.
\end{aligned}
\end{equation*}
We can now express the terms in  \eqref{K_Terms_Conv_Orig}  as
\begin{equation}\label{K_Terms_two}\notag
	\begin{aligned}
		& \kappa\int_0^T\!\!\int_{\TT} S_1 + S_2 + S_3 \di x \di t
	\end{aligned}
\end{equation}
with \begin{align*}
	S_1 :=
	& \frac{\rho^h\Dhmy\rho^h(\cdot+h_x\be_x)+\rho^h(\cdot+h_x\be_x)\Dhmy\rho^h}{2}\Dhmy\Dhpx\varphi_x
	\\
	&+\frac{\rho^h\Dhmx\rho^h(\cdot+h_y\be_y)+\rho^h(\cdot+h_y\be_y)\Dhmx\rho^h}{2}\Dhmx\Dhpy\varphi_y
	\\
	&+ \rho^h\Dhcx\rho^h\Dhmx\Dhpx\varphi_x+ \rho^h\Dhcy\rho^h\Dhmy\Dhpy\varphi_y,
	\\
	S_2 :=
	& \demi (\Dhpx\rho^h)^2\Dhpx\varphi_x
	+ \demi (\Dhpy\rho^h)^2\Dhpy\varphi_y
	\\
	&-\demi\Dhmy\rho^h\Dhmy \rho^h(\cdot+h_x\be_x)\Dhpx\varphi_x
	{-\demi\Dhmx\rho^h\Dhmx \rho^h(\cdot+h_y\be_y)\Dhpy\varphi_y}
	\\			
	&+\Dhmy\rho^h \Dhmy\rho^h(\cdot+h_x\be_x)\Dhpx\varphi_x(\cdot-h_y\be_y)
	\\
	&+\Dhmx\rho^h \Dhmx\rho^h(\cdot+h_y\be_y) \Dhpy\varphi_y(\cdot-h_x\be_x),
	\\
	S_3:=
	&\Dhpx\rho^h \Dhmx\rho^h \Dhcx\varphi_x
	+\Dhpy\rho^h \Dhmy\rho^h \Dhcy\varphi_y
	\\
	&+\Dhcx\rho^h\, \Dhpy \rho^h\Dhpy\varphi_x
	+\Dhcy\rho^h\, \Dhpx \rho^h\Dhpx\varphi_y.
	&
\end{align*}

In the following, we  investigate the integrals over $\mathrm{S}_i$ separately to demonstrate that terms in \eqref{K_Terms_two} tend to the desired limit terms in \eqref{NSK-consistency-momentum} with additional errors which vanish under grid refinement.
We rewrite the integral over $S_1$ using the identities
\begin{align}	
	D^\pm_{h,x}\rho^h 
	=& \, \Dhcx \rho^h \pm \frac{h_x}{2} \Dhmx \Dhpx \rho^h
	\\[0.5cm]
	\rho^h(\cdot\pm h_x\be_x)
	=& \, \rho^h  +  h_x D^\pm_{h,x} \rho^h
\end{align} and their analogues in $y$-direction to obtain
\begin{align*}
	\int_0^T\!\!\int_{\TT}\mathrm{S}_1\di x \di t = \int_0^T& \!\!\int_{\TT} \Big[\rho^h\Dhcx\rho^h-\rho^h \frac{h_x}{2} \Dhmx \Dhpx \rho^h +\rho^h \frac{h_y}{2}   \Dhpy\Dhmx\rho^h
	\\
	&\quad
	+\frac{h_y}{2}\Dhmx\rho^h\Dhpy\rho^h\Big]\Dhmy\Dhpx\varphi_x
	\\
	&+\Big[\rho^h\Dhcy\rho^h- \rho^h \frac{h_y}{2} \Dhmy \Dhpy \rho^h +\rho^h\frac{h_x}{2}   \Dhpx\Dhmy\rho^h
	\\
	&\quad +\frac{h_x}{2}\Dhmy\rho^h\Dhpx\rho^h\Big]\Dhmx\Dhpy\varphi_y
	\\
	&+ \rho^h\Dhcx\rho^h\Dhmx\Dhpx\varphi_x+ \rho^h\Dhcy\rho^h\Dhmy\Dhpy\varphi_y \di x \di t.
\end{align*}
For $\phi \in C^k, \, k \geq 2$  and its projection to cell averages $\phi_h$, we have:
\begin{align}\label{eq:phi_approx}
	D^\bullet_{h,x} \phi^h - \partial_x \phi = \mathcal{O}(h), 
	\quad D^\bullet_{h,y} \phi^h - \partial_y \phi  = \mathcal{O}(h).
\end{align}
%

We can also deduce the a-prori bounds
\begin{equation}\label{eq:limitL2L2Drho}
\norm{D^{\bullet}_h\rho^h}_{L^2(0,T;L^2(\mathbb{T}))} \leq C, \txt{and} h\norm{D^2_h\rho^h}_{L^2(0,T;L^2(\TT))} \to 0 \txt{as} h\to 0.
\end{equation}
from the energy inequality \eqref{eq:energydissipation} integrated over time.
Note that with \\
$\norm{D^{2}_h\rho^h}_{L^2(0,T;L^2(\mathbb{T}))}$, mixed differences in $x-$ and $y-$direction such as  \\
$\Dhmy \Dhpx \rho^h$ are also bounded, as a simple calculation shows. 
Using \eqref{eq:phi_approx}, we obtain for the integral over $S_1$:
\begin{align*}
\int_0^T\!\!\int_{\TT}\mathrm{S}_1\di x \di t
=&
\int_0^T\!\!\int_{\TT}
 \rho^h D^c_h \rho^h \cdot \nabla \operatorname{div} \boldsymbol{\varphi}  \di x\di t \\
&{+ \int_0^T\!\!\int_{\TT}  \rho^h D^c_h \rho^h \cdot \begin{pmatrix}
	1 \\
	1 \end{pmatrix} \mathcal{O}(h) \di x \di t}
\\
&+\int_0^T\!\!\int_{\TT}\Big[ -\rho^h \frac{h_x}{2} \Dhmx \Dhpx \rho^h +\rho^h\frac{h_y}{2}   \Dhpy\Dhmx\rho^h
\\
&\qquad \quad +\frac{h_y}{2}\Dhmx\rho^h\Dhpy\rho^h\Big]\Dhmy\Dhpx\varphi_x
\\
& \qquad + \Big[-\rho^h \frac{h_y}{2} \Dhmy \Dhpy \rho^h +\rho^h\frac{h_x}{2}   \Dhpx\Dhmy\rho^h
\\
&\qquad\quad +\frac{h_x}{2}\Dhmy\rho^h\Dhpx\rho^h\Big]\Dhmx\Dhpy\varphi_y \di x \di t.
\end{align*}

With the above bounds and Hölder's inequality, we find for the first error integral:
\begin{align*}
&\int_0^T \left| \int_{\TT}  \rho^h D^c_h \rho^h \cdot \begin{pmatrix}1 \\1 \end{pmatrix} \mathcal{O}(h) \di x \right| \di t =\int_0^T{ \|\rho^h\|_{L^2}}\|D_h^c \rho^h\|_{L^2} \mathcal{O}(h) \di t.
\end{align*}For the second error integral, we find
\begin{align*}
	&\int_0^T\Bigg| \int_{\TT}\Big[ -\rho^h \frac{h_x}{2} \Dhmx \Dhpx \rho^h +\rho^h\frac{h_y}{2}  \Dhpy\Dhmx\rho^h\Big] \Dhmy\Dhpx\varphi_x 
	\\
	&\qquad + \Big[-\rho^h \frac{h_y}{2} \Dhmy \Dhpy \rho^h +\rho^h\frac{h_x}{2}    \Dhpx\Dhmy\rho^h\Big]\Dhmx\Dhpy\varphi_y \di x \Bigg| \di t \\
	\lesssim & \, h\int_0^T { \|\rho^h\|_{L^2}}  \left\| D^2_h\rho^h \right\|_{L^2} \left\| \varphi \right\|_{C^2} \di t
\end{align*}and
\begin{align*}
&\int_0^T\Bigg| \int_{\TT} \frac{h_y}{2}\Dhmx\rho^h\Dhpy\rho^h \, \Dhmy\Dhpx\varphi_x + \frac{h_x}{2}\Dhmy\rho^h\Dhpx\rho^h \, \Dhmx\Dhpy\varphi_y \di x \Bigg| \di t
\\
\lesssim &\, h \int_0^T \|D^+_h\rho^h\|_{L^2}\|D^-_h\rho^h\|_{L^2}\left\| \varphi \right\|_{C^2} \di t;
\end{align*}
so that with the space-time bounds in \eqref{eq:limitL2L2Drho} we get 
\begin{align*}
\int_0^T\!\!\int_{\TT}\mathrm{S}_1\di x \di t
=&
\int_0^T\!\!\int_{\TT}
\rho^h D^c_h \rho^h \cdot \nabla \operatorname{div} \boldsymbol{\varphi}  \di x + e_1^h(t, \boldsymbol\varphi) \di t
\end{align*} 
with $\|e_1^h(t,\varphi)\|_{L^1(0,T)}\to 0$ for $t \to 0$, concluding the analysis of $S_1$.\\


\noindent Applying shift operators to rewrite $S_2$ in terms of central differences
and ordering error terms in the same manner as above, the integral over $S_2$ reads
\begin{align*}
&\int_0^T\!\!\int_{\TT}\mathrm{S}_2\di x \di t
\\
=& \int_0^T\!\!\int_{\TT}\demi (\Dhcx\rho^h  + \frac{h_x}{2} \Dhmx \Dhpx \rho^h)^2 \Dhpx\varphi_x
    \\
    &+ \demi (\Dhcy\rho^h  + \frac{h_y}{2} \Dhmy \Dhpy \rho^h)^2 \Dhpy\varphi_y
    \\
    &
    -\demi\left( \Dhcy\rho^h - \frac{h_y}{2} \Dhmy \Dhpy \rho^h\right)  
    \\
    &\quad
    \left( \Dhcy\rho^h - \frac{h_y}{2} \Dhmy \Dhpy \rho^h + h_x \Dhpx \Dhmy\rho^h \right)\Dhpx\varphi_x
     \\
    &
    -\demi\left( \Dhcx\rho^h  - \frac{h_x}{2} \Dhmx \Dhpx \rho^h\right) 
    \\
    &\quad
    \left( \Dhcx\rho^h - \frac{h_x}{2} \Dhmx \Dhpx \rho^h + h_y \Dhpy \Dhmx\rho^h \right)\Dhpy\varphi_y
    \\
    &
    +\left( \Dhcy\rho^h - \frac{h_y}{2} \Dhmy \Dhpy \rho^h\right)  \left( \Dhcy\rho^h - \frac{h_y}{2} \Dhmy \Dhpy \rho^h + h_x \Dhpx \Dhmy\rho^h \right)
    \\
    &\quad\left( \Dhpx\varphi_x - h_y\Dhmy\Dhpx\varphi_x \right) 
    \\
    &
    +\left( \Dhcx\rho^h  - \frac{h_x}{2} \Dhmx \Dhpx \rho^h\right) \left( \Dhcx\rho^h - \frac{h_x}{2} \Dhmx \Dhpx \rho^h + h_y \Dhpy \Dhmx\rho^h \right) 
     \\
    &\quad\left( \Dhpy\varphi_y - h_x\Dhmx\Dhpy\varphi_y \right)\di x \di t
    \\
   =& \int_0^T\!\!\int_{\TT}
   \frac12 |D_h^c \rho^h|^2 \,\mathbf{I} \,: \nabla \boldsymbol{\varphi} \di x \di t
   \\
   &  + \int_0^T\!\!\int_{\TT}
  |D_h^c \rho^h|^2 \,  \mathcal{O}(h) \di x \di t
  \\
  & + \int_0^T \!\!\int_{\TT} R^h \di x \di t,
\end{align*}
where $R^h$ consists of products of first and second order differences of $\rho_h, \varphi_x, \varphi_y$, and at least one of $h_x$ or $h_y$ taken to the first or second power. All the error terms can be treated in the same way as those arising for $S_1$ using the bounds \eqref{eq:limitL2L2Drho}, so that we get 
\begin{align*}
	\int_0^T\!\!\int_{\TT}\mathrm{S}_2\di x \di t
	=&
	\int_0^T\!\!\int_{\TT}
	 \frac12 |D_h^c \rho^h|^2 \,\mathbf{I} \,: \nabla \boldsymbol{\varphi}  \di x + e_2^h(t, \boldsymbol\varphi) \di t
\end{align*} with $\|e_2^h(t,\varphi)\|_{L^1(0,T)}\to 0$ for $t \to 0$.

\noindent Finally, we write 
\begin{align*}
	&\int_0^T\!\!\int_{\TT}\mathrm{S}_3\di x \di t
	\\
	=&\int_0^T\!\!\int_{\TT}
	\left( \Dhcx\rho^h  + \frac{h_x}{2} \Dhmx \Dhpx \rho^h\right) \left( \Dhcx\rho^h  - \frac{h_x}{2} \Dhmx \Dhpx \rho^h\right) \Dhcx\varphi_x
	\\
	&\qquad
	+\left( \Dhcy\rho^h  + \frac{h_y}{2} \Dhmy \Dhpy \rho^h\right) \left( \Dhcy\rho^h  - \frac{h_y}{2} \Dhmy \Dhpy \rho^h\right) \Dhcy\varphi_y
	\\
	&\qquad
	+\Dhcx\rho^h\, \left( \Dhcy\rho^h  + \frac{h_y}{2} \Dhmy \Dhpy \rho^h\right)\Dhpy\varphi_x
	\\
	&\qquad
	+\Dhcy\rho^h\, \left( \Dhcx\rho^h  + \frac{h_x}{2} \Dhmx \Dhpx \rho^h\right)\Dhpx\varphi_y \di x \di t
\end{align*}
and proceed as above to find that 
\begin{align*}
	\int_0^T\!\!\int_{\TT}\mathrm{S}_3\di x \di t
	=&
	\int_0^T\!\!\int_{\TT}
D_h^c \rho^h \otimes D_h^c \rho^h \,: \nabla \boldsymbol{\varphi}  \di x + e_3^h(t, \boldsymbol\varphi) \di t
\end{align*}holds with $\|e_3^h(t,\varphi)\|_{L^1(0,T)}\to 0$ for $t \to 0$, which concludes the proof.
\end{proof}
Now we can proceed to show the weak convergence of approximations to dissipative weak solutions. To this end, we will  make use of auxiliary results from \cite{eiter2026}:

\begin{lemma}\label{lem:conv}
	Let $\kappa > 0, \mu \geq 0$, and let the initial condition to the system \eqref{method:semi2d} $(\rho_0,\bu_0)	\colon \TT \to \real^+ \times \real^2$ be measurable with $\mathcal E_{\text{NSK}}(\rho_0,\bu_0) < \infty$. Suppose that in addition to Assumption \ref{assump},
	\begin{equation}
		\label{eq:lambda.limit}
		\lim_{h\to 0}\lambda h 
		=\lim_{h\to 0} \frac{h}{\lambda}=0
	\end{equation}holds.
	Then there exist $(\rho,\bm)$ such that $\rho$ has a weak derivative $\nabla \rho$ in space, and  
	\begin{subequations}
		\label{eq:conv.ek}
		\begin{align}
			\rho^h&\xrightharpoonup{*}\rho &&\text{in } L^\infty(0,T;L^\gamma(\TT)),
			\label{eq:conv.rho.ws}
			\\        
			D^\bullet_h \rho^h&\rightharpoonup\nabla\rho &&\text{in } L^\infty(0,T;L^2(\TT;\real^2)),
			\label{eq:conv.drho.w}
			\\
			\rho^h&\to\rho&&\text{in } C([0,T];L^p(\TT)) \text{ for all }p\in[1,\infty),
			\label{eq:conv.rho.s}
			\\
			\rho^h \bu^h &\xrightharpoonup{*}\bm &&\text{in } L^\infty(0,T;L^{\frac{2\gamma}{\gamma+1}}(\TT;\real^2)),
			\label{eq:conv.rhou.ws}
		\end{align}
	\end{subequations}
	hold (up to subsequences) as $h\to 0$.\\
	If $\mu$ is strictly positive, then there exists $\bu\in L^2(0,T;H^1(\TT;\real^2))$ such that $\bm=\rho \bu$ and 
	\begin{subequations}
		\label{eq:conv.nsk}
		\begin{align}
			\bu^h&\rightharpoonup\bu &&\text{in } L^2(0,T; L^2(\TT;\real^2)),
			\label{eq:conv.u.w}
			\\
			D_h^+ \bu^h&\rightharpoonup \nabla \bu \qquad &&\text{in } L^2(0,T;L^2(\TT;\real^{2\times 2})),
			\label{eq:conv.du.w}   
		\end{align}
	\end{subequations}
	also hold up to subsequences as $h \to 0$.
\end{lemma}
\begin{theorem}\label{thm:conv}
	Suppose $(\ro^h, \yu^h)$ is a sequence of approximations to the NSK system generated by the method \eqref{semi2drho} on grids of size $h_n \to 0$ with the initial condition $(\rho_0,\bu_0)	\colon \TT \to \real^+ \times \real^2$ measurable with $\mathcal E_{\text{NSK}}(\rho_0,\bu_0) < \infty$, and let Assumption \ref{assump} as well as \eqref{eq:lambda.limit} hold. Then there are subsequences so that
	\begin{align*}
		&\rho^h \to \rho \in L^\infty(0, T; L^\gamma(\mathbb{T}^2)), \\
		&D_h^\bullet \rho^h \rightharpoonup \nabla \rho \in L^\infty(0, T; L^2(\mathbb{T}^2, \real^2)),\\ 
		&\yu^h \convws \yu \in L^\infty(0,T;L^2(\mathbb{T}^2, \real^2))\end{align*}
	hold. For $\mu > 0$,
	$$
	D_h^+ \yu^h \rightharpoonup \nabla \yu \in L^\infty(0, T; L^2(\mathbb{T}^2, \real^{2\times 2}))
	$$ also holds. The limit $(\rho, \yu)$ is a dissipative weak solution to the system  \eqref{eq_Korteweg} with suitable defect measures.\\ 
\end{theorem}
\begin{proof}
	The existence of the required limits is given by Lemma \ref{lem:conv}. Note that for the non-viscous case $\mu = 0$,  we can just define $\yu$ as the quotient $\m / \rho$ almost everywhere on $[0,T] \times \TT$. 
	In passing to the limit towards the equations \eqref{eq:dw_continuity} -- \eqref{eq:dw_energy}, the proof proceeds along the lines of the treatment of the Navier-Stokes system in Chapter 5 in \cite{feireisl2021numerics}, and as above, we will omit the treatment of the continuity equation and of the Navier-Stokes terms in the momentum and energy equation already performed therein.  \\
	To pass to the limit of \eqref{eq:dw_momentum} from the consistency equation \eqref{NSK-consistency-momentum}, it then remains to show that:
	\begin{itemize}
		\item We have the weak convergence (up to subsequences)
		$\rho^h D^c_h\rho^h  \rightharpoonup \rho  \nabla\rho
		 \in L^{\infty}(0,T;L^1(\TT;\mathbb{R}^{2}))$;
		
		\item The limit (up to subsequences)
		\begin{align*}
		D^c_h\rho^h \otimes D^c_h\rho^h
		+ \frac12 |D^c_h\rho^h|^2 \, \mathbf I \convws& \,\overline{\nabla\rho \otimes \nabla\rho
			+ \frac12 |\nabla\rho|^2 \, \mathbf I} 
			\\
			&\in L^{\infty}(0,T;
		\mathcal{M}(\TT;\mathbb{R}^{2\times 2}))
		\end{align*}
			exists, and for the associated defect we have
			\begin{align*}
			\mathfrak{R}_K \coloneqq&  \overline{\nabla\rho \otimes \nabla\rho + \frac12 |\nabla\rho|^2 \, \mathbf I} - \left( \nabla\rho \otimes \nabla\rho + \frac12 |\nabla\rho|^2 \, \mathbf I\right)  
			\\
			&\in L^{\infty}(0,T;	\mathcal{M^+}(\TT;\mathbb{R}^{2\times 2}_{\text{sym}})).
		\end{align*}

	\end{itemize}
	Indeed, the weak convergence of $\rho^h D^c_h\rho^h$ follows directly from the strong convergence of $\rho^h$ and weak convergence of $D^c_h\rho^h$, both in $L^{\infty}(0,T;L^2(\TT;\mathbb{R}^{2}))$.\\
	The existence of the weak-$\ast$ limits follows by the Banach-Alaoglu theorem since $D^c_h\rho^h \otimes D^c_h\rho^h$ and $\frac12 |D^c_h\rho^h|^2 \, \mathbf I$ are bounded in $L^1(\TT)$, as was shown in the consistency proof above. From the $L^\infty L^1$ bounds we can also deduce the existence of a family of Young measures $ \mathcal{V}  = \lbrace \mathcal{V}\rbrace_{(t,x) \in [0,T]\times \TT} \in L^\infty_{w-\ast}([0,T]\times \TT, \mathcal{P}(\real\times\real^2))$ so that we have $\rho =
	\left\langle \mathcal{V}; \tilde \rho\right\rangle$, $\nabla\rho =  \left\langle \mathcal{V}; \ppsi \right\rangle$. We write the defect $\mathfrak{R}_K$ as the sum of the concentration and oscillation defects associated with that Young measure:

\begin{align*}	
\mathfrak{R}_K =& \left( \overline{\nabla\rho \otimes \nabla\rho
		+ \frac12 |\nabla\rho|^2 \, \mathbf I} - \left\langle \mathcal{V}; \ppsi \otimes \ppsi
		+ \frac12 |\ppsi|^2 \, \mathbf I\right\rangle \right) \\
		&+ \left( \left\langle \mathcal{V}; \ppsi \otimes \ppsi	+ \frac12 |\ppsi|^2 \, \mathbf I\right\rangle
		- \nabla\rho \otimes \nabla\rho		- \frac12 |\nabla\rho|^2 \, \mathbf I\right) \\
		=& \mathfrak{R}_{K,c}+\mathfrak{R}_{K,o}
\end{align*}
Observing that for any $\xi \in \real^2$, the function $\ppsi \mapsto \left( \xi \otimes \xi \right) :   \left( \ppsi \otimes \ppsi	+ \frac12 |\ppsi|^2\right)  $
is  lower semicontinuous and convex,  we can employ the same arguments as in Chapter 5.3  in \cite{feireisl2021numerics}
to deduce via the comparison principle for concentration defects and Jensen's inequality respectively that $\mathfrak{R}_{K,c}$ and $\mathfrak{R}_{K,o}$ both take values in the positive semidefinite matrices as required.\\[0.3cm]
To derive \eqref{eq:dw_energy}, it remains to show (with the viscous term already taken care of) that we can pass to the limit in the energy inequality \eqref{eq:energydissipation} (integrated over time) with a suitable defect, i.e:
$$\ENSK(\rho^h,\bu^h) \convws \overline{\ENSK(\rho,\bu)} \in  L^{\infty}(0,T;	\mathcal{M}(\TT;\mathbb{R}))$$
up to a subsequence, and
$$\mathfrak{E}'\coloneqq \overline{\ENSK(\rho,\bu)} - \ENSK(\rho,\bu) \in L^{\infty}(0,T;	\mathcal{M}^+(\TT;\mathbb{R})).
$$
Existence of $\overline{\ENSK(\rho,\bu)}$  follows from the energy stability of the scheme, and for the suitability of the defect $\mathfrak{E}'$ we can again argue as in \cite{feireisl2021numerics}, splitting $\mathfrak{E}'$ into a concentration and oscillation defect associated with a Young measure, since with the addition of the capillary term $\frac{\kappa}{2}|\nabla \rho|^2$, $\ENSK$ remains continuous and convex as a function of the variables $\rho, \nabla \rho,$ and $\rho \yu$.\\
Finally, the compatibility of the newly introduced defects has to be shown, i.e. 
$$c\,\mathfrak{E}' \leq \operatorname{tr}[\mathfrak{R_1}]\leq C\,\mathfrak{E}' \txt{for some $0<c\leq C$.}$$
This however is obvious from $\operatorname{tr}\left[ \nabla\rho \otimes \nabla\rho + \frac12 |\nabla\rho|^2 \, \mathbf I\right] = 2|\nabla\rho|^2$.

\end{proof}

\small
\bibliographystyle{abbrv}
\bibliography{literature}
\end{document}